\documentclass{amsart}
\usepackage{amssymb}
\usepackage{hyperref}

\newcommand{\w}{\omega}
\newcommand{\F}{\mathcal F}
\newcommand{\Ra}{\Rightarrow}

\newcommand{\I}{\mathcal I}
\newcommand{\A}{\mathcal A}
\newcommand{\C}{\mathcal C}
\newcommand{\K}{\mathcal K}
\newcommand{\PP}{\mathcal P}
\newcommand{\cbox}{\boxdot}

\newtheorem{theorem}{Theorem}[section]
\newtheorem{corollary}[theorem]{Corollary}
\newtheorem{lemma}[theorem]{Lemma}
\newtheorem{example}[theorem]{Example}
\newtheorem{proposition}[theorem]{Proposition}

\theoremstyle{definition}
\newtheorem{definition}[theorem]{Definition}

\title[The strong Pytkeev$^*$ property]{The strong Pytkeev$^*$ property of topological spaces}
\author{Taras Banakh}
\address{Ivan Franko National University of Lviv (Ukraine) and Institute of Mathematics, Jan Kochanowski University in Lviv (Ukraine)}
\email{t.o.banakh@gmail.com}

\begin{document}

\begin{abstract}  Modifying the known definition of a Pytkeev network, we introduce  a notion of  Pytkeev$^*$ network and prove that a topological space has a countable Pytkeev network if and only if $X$ is countably tight and has a countable Pykeev$^*$ network at $x$. In the paper we establish some stability properties of the class of topological spaces with the strong Pytkeev$^*$-property.
\end{abstract}
\maketitle

\section{Introduction}

In this paper we introduce and study the strong Pytkeev$^*$ property, which is a countable modification of the strong Pytkeev property. The latter property was introduced in \cite{TZ} and studied in \cite{BL}, \cite{GartMor}. Our interest to the strong Pytkeev$^*$ property is motivated by a recent result of the first author who proved in \cite{Ban} that each topological space with a local $\w^\w$-base has the strong Pytkeev$^*$-property.
A topological space $X$ is defined to have a {\em local $\w^\w$-base} if each point $x\in X$  has a neighborhood base $(U_\alpha)_{\alpha\in\w^\w}$ such that $U_\beta\subset U_\alpha$ for all $\alpha\le\beta$ in $\w^\w$. Here the set  $\w^\w$ of all functions from $\w$ to $\w$ is endowed with the partial order $\le $ defined by $\alpha\le\beta$ iff $\alpha(n)\le\beta(n)$ for all $n\in\w$. By $\w$ we denote the smallest infinite cardinal.


By definition, a topological space $X$ has the strong Pytkeev$^*$ property  if $X$ has a countable Pytkeev$^*$ network at each point $x\in X$. Pytkeev$^*$ networks are introduced and studied in Section~\ref{s:P*}. Pytkeev$^*$ networks are modifications of $\w$-Pytkeev networks which are particular cases of $\kappa$-Pytkeev networks, considered in Section~\ref{s:kP}.

\section{Pytkeev networks versus $\kappa$-Pytkeev networks}\label{s:kP}

In this section we recall the definition of a Pytkeev network and inserting a cardinal parameter $\kappa$ in this definition introduce $\kappa$-Pytkeev network. Then we study the interplay between these two notions.

For a subset $A$ of a topological space $X$ by $\bar A$ we denote the closure of $A$ in $X$. We shall say that a subset $A$ of a topological space $X$ {\em accumulates} at a point $x\in X$ is each neighborhood $O_x\subset X$ of $x$ contains infinitely many points of the set $A$.

\begin{definition}\label{d:Pytkeev} A family $\mathcal N$ of subsets of a topological space $X$ is called
\begin{itemize}
\item a {\em network} if for each point $x\in X$ and neighborhood $O_x\subset X$ there exists a set $N\in\mathcal N$ such that $x\in N\subset O_x$;
\item a {\em Pytkeev network} at a point $x\in X$ if for each neighborhood $O_x\subset X$ of $x$ and each subset $A\subset X$ with $x\in \bar A$ there is a set $N\in\mathcal N$ such that $x\in N\subset X$, $N\cap A\ne\emptyset$, and moreover $N\cap A$ is infinite if the set $A$ accumulates at $x$;
\item a {\em Pytkeev network} if $\mathcal N$ is a Pytkeev network at each point $x\in X$.
\end{itemize}
\end{definition}
Pytkeev networks were introduced in \cite{Ban} and thoroughly studied in \cite{Ban}, \cite{BL}, \cite{GK15b}.

Inserting a restriction on the cardinality of the sets $A$ in Definition~\ref{d:Pytkeev}, we obtain the definition of a $\kappa$-Pytkeev network.

\begin{definition}\label{d:Pytkeev} Let $\kappa$ be a cardinal. A family $\mathcal N$ of subsets of a topological space $X$ is called
\begin{itemize}
\item a {\em $\kappa$-Pytkeev network} at a point $x\in X$ if for each neighborhood $O_x\subset X$ of $x$ and each subset $A\subset X$ of cardinality $|A|\le\kappa$ with $x\in \bar A$ there is a set $N\in\mathcal N$ such that $x\in N\subset X$, $N\cap A\ne\emptyset$, and moreover $N\cap A$ is infinite if the set $A$ accumulates at $x$;
\item a {\em $\kappa$-Pytkeev network} if $\mathcal N$ is a $\kappa$-Pytkeev network at each point $x\in X$.
\end{itemize}
\end{definition}

The {\em tightness} $t_x(X)$ of a topological space $X$ at a point $x\in X$ is defined as the smallest infinite cardinal $\kappa$ such that each subset $A\subset X$ with $x\in\bar A$ contains a subset $B\subset A$ of cardinality $|B|\le\kappa$ with $x\in\bar B$.

\begin{proposition}\label{p:P<=>kP+t} For a family $\mathcal N$ of subsets of a topological space $X$ and a point $x\in X$ the following conditions are equivalent:
\begin{enumerate}
\item $\mathcal N$ is a Pytkeev network at $x$;
\item $\mathcal N$ is a $\kappa$-Pytkeev network at $x$ for every cardinal $\kappa$;
\item $\mathcal N$ is a $t_x(X)$-Pytkeev network at $x$.
\end{enumerate}
\end{proposition}

\begin{proof} The implications $(1)\Leftrightarrow(2)\Ra(3)$ are trivial. To prove that $(3)\Ra(1)$, assume that $\mathcal N$ is a $t_x(X)$-Pytkeev network at $x$. Given any neighborhood $O_x\subset X$ of $x$ and a subset $A\subset X$ with $x\in\bar A$, we need to find a set $N\in\mathcal N$ such that $x\in N\subset O_x$, $N\cap A\ne\emptyset$ and $N\cap A$ is infinite if $A$ accumulates at $x$. If the set $A$ does not accumulate at $x$, then there exists a neighborhood $V_x\subset O_x$ such that $V_x\cap A$ is finite. Replacing $V_x$ by a smaller neighborhood we can assume that $V_x\cap A=\ddot x\cap A$ where $\ddot x$ is the intersection of all neighborhoods of $x$ in $X$. Since $x\in \bar A$, the intersection $\ddot x\cap A=V_x\cap A$ is not empty, which allows us to find a singleton $B\subset \ddot x\cap A$ and observe that $x\in\bar B$. Since $|B|=1\le t_x(X)$ and $\mathcal N$ is a $t_x(X)$-Pytkeev network at $x$, there exists a set $N\in\mathcal N$ such that $x\in N\subset O_x$ and $\emptyset\ne N\cap B\subset N\cap A$.

Now assume that the set $A$ accumulates at $x$. If $A\cap\ddot x$ is infinite, then choose any countable infinite set $B\subset A\cap\ddot x$ and observe that $B$ accumulates at $x$. Since $|B|=\w\le t_x(X)$ the $t_x(X)$-Pytkeev network $\mathcal N$ contains a set $N$ such that $x\in N\subset O_x$ and $N\cap B\subset N\cap A$ is infinite. If $\ddot x$ is finite, then the set $A\setminus\ddot x$ accumulates at $x$. By the definition of the cardinal $t_x(X)$, there exists a subset $B\subset A\setminus \ddot x$ of cardinality $|B|\le t_x(X)$ such that $x\in\bar B$. Since $B\cap\ddot x=\emptyset$, the set $B$ accumulates at $x$. Then the $t_x(X)$-Pytkeev network $\mathcal N$ contains a set $N$ such that $x\in N\subset O_x$ and $N\cap B\subset N\cap A$ is infinite.
\end{proof}

In fact, the tightness $t_x(X)$ of a topological space $X$ at a point $x$ is upper bounded by the cardinality of a Pytkeev network at $x$.

\begin{lemma}\label{l:P=>t} If $\mathcal N$ is a Pytkeev network at a point $x$ of a topological space $X$, then $t_x(X)\le\w\cdot|\mathcal N|$.
\end{lemma}

 \begin{proof}  Given any subset  $A\subset X$ with $x\in\bar A$, we should will find a subset $B\subset A$ of cardinality $|B|\le |\mathcal N|$ such that $x\in\bar B$. Consider the subfamily $\mathcal N'=\{N\in\mathcal N:N\cap A\ne\emptyset\}$ and for every $N\in\mathcal N'$ choose a point $x_N\in A\cap N$. It is clear that the set $B=\{x_N:N\in\mathcal N'\}$ has cardinality $|B|\le|\mathcal N|$. It remains to prove that $x\in\bar B$. By Definition~\ref{d:Pytkeev}, for any neighborhood $O_x\subset X$ of $x$ we can find a set $N\in\mathcal N$ such that
$x\in N\subset O_x$ and $N\cap A\ne\emptyset$. Then $N\in\mathcal N'$ and hence $x_N\in B\cap O_x$; so $B\cap O_x$ is not empty and hence $x\in\bar B$.
\end{proof}

A topological space $X$ is {\em countable tight} at a point $x\in X$ if $t_x(X)=\w$. A space $X$ is {\em countably tight} if it is countably tight at each point.
Proposition~\ref{p:P<=>kP+t} implies:

\begin{corollary}\label{c:ct+P=>wP} A family $\mathcal N$ of subsets of a countably tight space $X$ is a Pytkeev network at a point $x\in X$ if and only if $\mathcal N$ is an $\w$-Pytkeev network at $x$.
\end{corollary}

Combining Corollary~\ref{c:ct+P=>wP} with Lemma~\ref{l:P=>t}, we get another corollary.

\begin{corollary}\label{c:P<=>wP} A countable family $\mathcal N$ of subsets of a topological space $X$ is a Pytkeev network at a point $x\in X$ if and only if the space $X$ is countably tight at $x$ and $\mathcal N$ is an $\w$-Pytkeev network at $x$.
\end{corollary}

A family $\mathcal N$ of subsets of a topological space $X$ is called {\em point-countable} if for each $x\in X$ the family $\mathcal N_x=\{N\in\mathcal N:x\in N\}$ is at most countable. Corollary~\ref{c:P<=>wP} implies:

\begin{corollary}\label{c:pcP<=>wP} A point-countable family $\mathcal N$ of subsets of a topological space is a Pytkeev network if and only if $X$ is countably tight and $\mathcal N$ is an $\w$-Pytkeev network.
\end{corollary}

\section{Pytkeev$^*$-networks versus $\w$-Pytkeev networks}\label{s:P*}

The definition of a $\kappa$-Pytkeev network is a bit complicated as it actually contains two requirements (for non-accumulating and accumulating sets $A$). In $T_1$-spaces $\w$-Pytkeev networks can  be equivalently defined using accumulating sequences in place of countable accumulating sets. This yields a more simple and a more natural notion of a Pytkeev$^*$ network.

We say that a sequence $(x_n)_{n\in\w}$ of points of a topological space $X$ {\em accumulates} at a point $x\in X$ if each neighborhood $O_x\subset X$ of $x$ contains  infinitely many points $x_n$, $n\in\w$.

\begin{definition}\label{d:Pytkeev*} A  family $\mathcal N$ of subsets of a topological space $X$ is called
\begin{itemize}
\item a {\em Pytkeev$^*$ network} at a point $x\in X$ if for each neighborhood $O_x\subset X$ of $x$ and each sequence $(x_n)_{n\in\w}\in X^\w$ accumulating at $x$ there is a set $N\in\mathcal N$ such that $x\in N\subset O_x$ and $N$ contains infinitely many points $x_n$, $n\in\w$.
\item a {\em Pytkeev$^*$ network} if $\mathcal N$ is a Pytkeev$^*$ network at each point $x\in X$.
\end{itemize}
\end{definition}

Pytkeev networks can be characterized using accumulating sequences of finite sets in place of accumulating sequences of points. We shall say that a sequence $(F_n)_{n\in\w}$ of subsets of a topological space $X$ {\em accumulates} at a point $x\in X$ if each neighborhood $O_x\subset X$ of $x$ intersects infinitely many sets $F_n$, $n\in\w$.

 \begin{proposition}\label{p:P*p<=>P*f} A family $\mathcal N$ of subsets of a topological space $X$ is a Pytkeev$^*$ network at a point $x\in X$ if and only if   for each neighborhood $O_x\subset X$ of $x$ and each sequence $(F_n)_{n\in\w}$ of finite subsets of $X$ accumulating at $x$ there is a set $N\in\mathcal N$ such that $x\in N\subset O_x$ and $N$ intersects infinitely many sets $F_n$, $n\in\w$.
\end{proposition}

\begin{proof} The ``if'' part is trivial. To prove the ``only if'' part, assume that $\mathcal N$ is a Pytkeev$^*$ network at $x$. Take any neighborhood $O_x\subset X$ of $x$ and any sequence $(F_n)_{n\in\w}$ of finite subsets of $X$ accumulating at $x$. Let $n_k=0$ and $n_{k+1}=n_k+|F_k|$ for $k\in\w$. Let $\{x_n\}_{n\in\w}$ be an enumeration of the union $\bigcup_{k\in\w}F_k$ such that $F_k=\{x_i:n_k\le i<n_{k+1}\}$ for all $k\in\w$. It is easy to see that the sequence $(x_n)_{n\in \w}$ accumulates at $x$. Since $\mathcal N$ is a Pytkeev$^*$ network, there exists a set $N\in\mathcal N$ such that $x\in N\subset O_x$ and $N$ contains infinitely many points $x_n$, $n\in\w$. Then $N$ intersects infinitely many sets $F_n$, $n\in\w$.
\end{proof}

Pytkeev$^*$ networks are tightly related to $\w$-Pytkeev networks.

\begin{lemma}\label{l:P*=>wP} If a family $\mathcal N$ of subsets of a topological space $X$ is a Pytkeev$^*$ network at a point $x\in X$, then $\mathcal N$ is an $\w$-Pytkeev network at $x$.
\end{lemma}

\begin{proof} Given a neighborhood $O_x\subset X$ of $x$ and a countable subset $A\subset X$ with $x\in\bar A$, we need to find a set $N\in\mathcal N$ such that $x\in N\subset O_x$, $N\cap A\ne\emptyset$ and moreover the intersection $N\cap A$ is infinite if $A$ accumulates at $x$.

If $A$ does not accumulate at $x$, then we can find a neighborhood $V_x\subset X$ of $x$ such that the intersection $V_x\cap A$ is finite. Replacing $V_x$ by a smaller neighborhood, we can assume that $V_x\cap A=\ddot x\cap A$, where $\ddot x$ is the intersection of all neighborhoods of $x$ in $X$. Choose any sequence $\{x_n\}_{n\in\w}\subset A\cap\ddot x$ and observe that this sequence accumulates at the point $x$. Since $\mathcal N$ is a Pytkeev$^*$ network at $x$, there is a set $N\in\mathcal N$ such that $x\in N\subset O_x$ and $N$ contains infinitely many points $x_n\in A$. In particular, $N$ intersects the set $A$.

  Now assume that the set $A$ accumulates at $x$. If $A\cap \ddot x$ is infinite, then we can choose a sequence $(x_n)_{n\in\w}$ of pairwise distinct points in $A\cap\ddot x$. It follows that the sequence $(x_n)_{n\in\w}$   accumulates at $x$. Since $\mathcal N$ is a Pytkeev$^*$ network at $x$, there exists a set $N\in\mathcal N$ such that $x\in N\subset O_x$ and $N$ contains infinitely many points $x_n$, $n\in\w$. Since these points are pairwise distinct, the intersection $N\cap A\cap\ddot x\subset N\cap A$ is infinite and we are done.

So, assume that $A\cap\ddot x$ is finite. Since the set $A$ accumulates at $x$, the complement $A\setminus\ddot x$ is also accumulates at $x$. So, $A\setminus\ddot x=\{x_n\}_{n\in\w}$ for some sequence $(x_n)_{n\in\w}$ of pairwise distinct points of $X$. Since the sequence $(x_n)_{n\in\w}$ accumulates at $x$ and $\mathcal N$ is a Pytkeev$^*$ network, there exists a set $N\in\mathcal N$ containing infinitely many points $x_n$, $n\in\w$. Since these points are pairwise distinct, the intersection $N\cap B\subset N\cap A$ is infinite, witnessing that $\mathcal N$ is an $\w$-Pytkeev network at $x$.
\end{proof}

Lemma~\ref{l:P*=>wP} and Corollary~\ref{c:ct+P=>wP} imply:

\begin{corollary}\label{c:ct+P*=>P} If a topological  space $X$ is countably tight at a point $x\in X$, then each Pytkeev$^*$-network at $x$ is a Pytkeev network at $x$.
\end{corollary}

For a point $x$ of a topological space $X$ by $\ddot x$ we denote the intersection of all neighborhoods of $x$ in $X$. It is clear that $X$ is a $T_1$-space if and only if $\ddot x=\{x\}$ for all $x\in X$.
For a subset $A$ of a topological space $X$ put $\ddot A=\bigcup_{x\in A}\ddot x$. Observe that $A\subset \ddot A$ and each open subset $U\subset X$ containing $A$ contains also $\ddot A$. If $X$ is a $T_1$-space, then $\ddot A=A$.

\begin{lemma}\label{l:wP=>P*} If a family $\mathcal N$ of subsets of a topological space $X$ is an $\w$-Pytkeev network at a point $x\in X$, then the family $\ddot{\mathcal N}=\{\ddot N:N\in\mathcal N\}$ is a Pytkeev$^*$ network at $x\in X$.
\end{lemma}

\begin{proof}  Given a neighborhood $O_x\subset X$ of $x$ and a sequence $(x_n)_{n\in\w}$ of points accumulating at $x$, we need to find a set $\ddot N\in\ddot{\mathcal N}$ containing infinitely points $x_n$, $n\in\w$. If $\ddot x$ contains infinitely many points $x_n$, $n\in\w$, then  take any set $N\in\mathcal N$ with $x\in N\subset O_x$ and conclude that the set $\ddot N\supset\ddot x$ contains infinitely many points $x_n$, $n\in\w$. If $\ddot x$ contains only finitely many points $x_n$, $n\in\w$, then the set $A=\{x_n\}_{n\in\w}$ accumulates at $x$ and we can find a set $N\in\mathcal N$ such that $x\in N\subset O_x$ and $N\cap A$ is infinite. Then the set $\ddot N\in\ddot{\mathcal N}$ is included in $O_x$ and contains infinitely many points $x_n$, $n\in\w$, witnessing that $\mathcal N$ is a countable Pytkeev$^*$ network at $x$.
\end{proof}

Lemmas~\ref{l:P*=>wP} and \ref{l:wP=>P*} imply the following characterization.

\begin{corollary}\label{c:T1P*<=>wP} A family $\mathcal N$ of subsets of a $T_1$-space $X$ is a Pytkeev$^*$ network at a point $x\in X$ if and only if $\mathcal N$ is an $\w$-Pytkeev network at $x$.
\end{corollary}

Corollary~\ref{c:P<=>wP} and Lemmas~\ref{l:P*=>wP}, \ref{l:wP=>P*} imply the following  characterizations.


\begin{corollary}\label{c:P=ct+P*} A topological space $X$ has a countable Pytkeev network at a point $x\in X$ if and only if $X$ has a countable Pytkeev$^*$ network at $x$ and $X$ is countably tight at $x$.
\end{corollary}

\begin{corollary}\label{c:wP<=>wP*} A topological space $X$ has a countable Pytkeev network if and only if $X$ has a countable Pytkeev$^*$ network.
\end{corollary}

\begin{proof} The ``only if'' part follows from Lemma~\ref{l:wP=>P*}. To prove the ``if'' part, assume that $X$ has a countable Pytkeev$^*$ network. Then $X$ has a countable network and hence is hereditarily separable and countably tight. By Corollary~\ref{c:ct+P*=>P}, the space $X$ has  a countable Pytkeev network.
\end{proof}

Finally we present an example of a Pytkeev$^*$-network which is not a Pytkeev network.

A point $x$ of a topological space $X$ is a {\em weak $P$-point} if $x\notin \bar A$ for any countable set $A\subset X\setminus\{x\}$. By \cite[4.3.4]{vM} the remainder $\beta\w\setminus \w$ of the Stone-\v Cech compactification $\beta\w$ of $\w$ contains a non-isolated weak $P$-point.

\begin{example} If $X$ is a non-isolated weak $P$-point in a topological space, then the family $\mathcal N=\big\{\{x\}\big\}$ is a Pytkeev$^*$ network at $x$, which is not a Pytkeev network at $x$.
\end{example}

\section{The strong Pytkeev$^*$ property versus the strong Pytkeev property}

In this section we recall the definition of the strong Pytkeev property, introduce the strong Pytkeev$^*$ property and study the interplay between these two notions.

 \begin{definition}\label{d:sPp} A topological space $X$ is defined to have
 \begin{itemize}
 \item  the {\em strong Pytkeev property at a point} $x\in X$ if $X$ has  a countable Pytkeev network at $x\in X$;
 \item  the {\em strong Pytkeev property} if $X$ has the strong Pytkeev property at each point $x\in X$.
 \end{itemize}
\end{definition}
The strong Pytkeev property was introduced in \cite{TZ} and thoroughly studied in \cite{MSak}, \cite{GKL14}, \cite{BL}.

Replacing in Definition~\ref{d:sPp} Pytkeev networks by Pytkeev$^*$ networks, we obtain  the definition of the strong Pytkeev$^*$ property.

 \begin{definition}\label{d:sPp}  A topological space $X$ is defined to have
 \begin{itemize}
 \item  the {\em strong Pytkeev$^*$ property at a point} $x\in X$ if $X$ has  a countable Pytkeev$^*$ network at $x\in X$;
 \item  the {\em strong Pytkeev$^*$ property} if $X$ has the strong Pytkeev$^*$ property at each point $x\in X$.
 \end{itemize}
\end{definition}

Corollary~\ref{c:P=ct+P*} implies that strong Pytkeev property decomposes into the combination of the countable tightness and the strong Pytkeev$^*$ property.

\begin{theorem}\label{t:P<=>P*+ct} A topological space $X$ has the strong Pytkeev property if and only if $X$ has the strong Pytkeev$^*$ property and $X$ is  countably tight.
\end{theorem}

In \cite[3.2]{MSak} Sakai observed that the first countability decomposes into two properties: the strong Pytkeev property and the countable fan tightness.

\begin{definition} A topological space $X$ is defined to have
{\em countable fan} ({\em ofan}) {\em tightness at a point} $x\in X$ if for every decreasing sequence $(A_n)_{n\in\w}$ of (open) subsets of $X$ with $x\in\bigcap_{n\in\w}\bar A_n$ there exist a sequence of finite sets $F_n\subset A_n$, $n\in\w$, such that each neighborhood of $x$ intersects infinitely many sets $F_n$, $n\in\w$.

A topological space $X$ is defined to have the {\em countable fan} ({\em ofan}) {\em tightness} if
$X$ has {\em countable fan} ({\em ofan}) {\em tightness}  at each point $x\in X$.
\end{definition}

The countable fan tightness was introduced by Arhangelskii \cite{Ar86} and is well-known and useful property in $C_p$-theory \cite{Arh}. Its ``open'' modification was introduced by Sakai \cite{MSak} as the property $(\#)$. He used this modification to prove the following characterization of the first countability (which can be also find in \cite[1.6, 1.7]{BL}). We recall that a topological space $X$ is {\em first-countable} at a point $x\in X$ if $X$ has a countable neighborhood base at $x$.

\begin{theorem}[Sakai]\label{t:MSak} A (regular) topological space $X$ is first countable at a point $x\in X$ if and only if $X$ simultaneously has countable fan (ofan) tightness at $x$ and the strong Pytkeev property at $x$.
\end{theorem}

In this characterization the strong Pytkeev property can be weakened to the strong Pytkeev$^*$ property and the regularity of the space $X$ can be weakened to the semiregularity of $X$.

A topological space $X$ is called {\em semiregular at a point} $x\in X$ if each neighborhood of $x$ contains the interior of the closure of some other neighborhood of $x$. A topological space is {\em semiregular} if it is semiregular at each point.

\begin{theorem}\label{t:1<=>P*+cft}  For a topological space $X$ and a point $x\in X$ the following conditions are equivalent:
\begin{enumerate}
\item $X$ has a countable neighborhood base at $x$;
\item $X$ has the strong Pytkeev$^*$ property at $x$ and has countable fan tightness at $x$.
\end{enumerate}
If the space $X$ is semiregular at $x$, then the conditions \textup{(1), (2)} are equivalent to
\begin{enumerate}
\item[(3)] $X$ has the strong Pytkeev$^*$ property at $x$ and the countable ofan tightness at $x$.
\end{enumerate}
\end{theorem}

\begin{proof} The implications $(1)\Ra(2)\Ra(3)$ are trivial. The implication $(2)\Ra(1)$ follow from Theorem~\ref{t:MSak}, Corollary~\ref{c:P=ct+P*} and the countable tightness of spaces with countable fan tightness.

To prove that $(3)\Ra(1)$ assume that $X$ is semiregular at $x$, $X$ has countable fan open-tightness at $X$,  and $X$ has a countable Pytkeev$^*$ network $\mathcal N$ at $x$. Replacing $\mathcal N$ by a largest countable family, we can assume that $\mathcal N$ is closed under finite unions. For every set $N\in \mathcal N$ let $\bar N^\circ$ be the interior of the closure $\bar N$ of the set $N$ in $X$. Consider the countable family $\mathcal B=\{\bar N^\circ:N\in\mathcal N\}$ of open sets in $X$. We claim that its subfamily $\mathcal B_x=\{B\in\mathcal B:x\in B\}$ is a neighborhood base at $x$.

Given a neighborhood $O_x\subset X$ of $x$, we should find a set $B\in \mathcal B_x$ such that $B\subset O_x$. Since the space $X$ is semiregular at $x$, the neighborhood $O_x$ contains the interior $\bar U_x^\circ$ of the closure $\bar U_x$ of some neighborhood $U_x$ of $x$.

Let $\{N_k\}_{k\in\w}$ be an enumeration of the countable subfamily $\mathcal N'=\{N\in\mathcal N:N\subset U_x\}$ and let $M_k$ be the closure of the set $\bigcup_{i\le k}N_i$ in $X$. We claim that for some $k\in\w$ the set $M_k$ is a neighborhood of $x$. Assuming that this not true, we conclude that for every $k\in\w$ the open set $A_k=X\setminus M_k$ contains the point $x$ in its closure.  By the countable ofan tightness of $X$ at $x$, there exists a sequence $(F_k)_{k\in\w}$ of finite subsets $F_k\subset A_k$, which accumulates at $x$. By Proposition~\ref{p:P*p<=>P*f}, the Pytkeev$^*$ network contains a set $N\in\mathcal N$ such that $x\in N\subset U_x$ and $N$ intersects infinitely many sets $F_k$, $k\in\w$. It follows that $N\in\mathcal N'$ and hence $N=N_k$ for some $k\in\w$. Since $N\subset M_n\subset X\setminus F_n$ for all $n\ge k$, the set $N$ cannot intersect the sets $F_n$ for $n\ge k$. This contradiction shows that for some $k\in\w$ the set $M_k$ is a neighborhood of $x$.  Since the family $\mathcal N$ is closed under finite unions, the set $N=\bigcup_{i\le k}N_i$ belongs to the family $\mathcal N$ and the interior $\bar N^\circ$ of the closure $\bar N=M_k$ of $N$ is a neighborhood of $x$ such that $\bar N^\circ \subset \bar U_x^\circ \subset O_x$. Since $\bar N^\circ\in\mathcal B_x$, the countable family $\mathcal B_x$ is a neighborhood base at $x$.
\end{proof}

A similar characterization holds for second countable spaces.

\begin{theorem}  For a topological space $X$ the following conditions are equivalent:
\begin{enumerate}
\item $X$ is second-countable;
\item $X$ has a countable Pytkeev$^*$ network and $X$ has countable fan tightness at $x$.
\end{enumerate}
If the space $X$ is semiregular, then the conditions \textup{(1), (2)} are equivalent to
\begin{enumerate}
\item[(3)] $X$ has a countable Pytkeev$^*$ network and $X$ has countable ofan tightness.
\end{enumerate}
\end{theorem}

\begin{proof} The implications $(1)\Ra(2)\Ra(3)$ are trivial and the implication $(2)\Ra(1)$ follows from Corollary~\ref{c:wP<=>wP*} and Theorem~1.12 \cite{Ban} (saying that a space $X$ is second-countable if $X$ has a countable Pytkeev network and countable fan tightness). To prove that $(3)\Ra(1)$, assume that the space $X$ is semiregular, has countable Pytkeev$^*$ network $\mathcal N$ and has countable fan open-tightness. It follows from the proof of Theorem~\ref{t:1<=>P*+cft} that the countable family $\mathcal B=\{\bar N^\circ:N\in\mathcal N\}$ is a base of the topology of $X$.
\end{proof}

\section{Applications of Pytkeev$^*$ networks to $\mathfrak P_0$ and $\mathfrak P$-spaces}

In this section we apply Pytkeev$^*$ networks to characterize $\mathfrak P_0$-spaces and $\mathfrak P$-spaces.

\begin{definition} A regular topological space $X$ is defined to be a
\begin{itemize}
\item a {\em $\mathfrak P_0$-space} if $X$ has a countable Pytkeev network;
\item a {\em $\mathfrak P$-space} if $X$ has a $\sigma$-locally finite Pytkeev network.
\end{itemize}
\end{definition}
$\mathfrak P_0$-spaces and $\mathfrak P$-spaces were introduced and studied in \cite{Ban} and \cite{GK15b}, respectively.   It is clear that each $\mathfrak P_0$-space is a $\mathfrak P$-space and each $\mathfrak P$-space has the strong Pytkeev property.

Corollaries~\ref{c:wP<=>wP*}, \ref{c:pcP<=>wP} and \ref{c:T1P*<=>wP} imply the following characterizations

\begin{theorem} A regular topological space $X$ is a $\mathfrak P_0$-space if and only if $X$  has a countable Pytkeev$^*$ network.
\end{theorem}

\begin{theorem} A regular topological space $X$ is a $\mathfrak P$-space if and only if $X$ is countably tight and has a $\sigma$-locally finite Pytkeev$^*$ network.
\end{theorem}

 By Theorem~4.5 \cite{GK15b}, a topological space $X$ is metrizable if and only if $X$ is a $\mathfrak P$-space of countable fan tightness. This result can be generalized in two direction.

 \begin{theorem} A regular topological space $X$ is metrizable if and only if $X$ has a $\sigma$-locally finite Pytkeev$^*$ network and $X$ has countable tightness and countable ofan tightness.
 \end{theorem}

 \begin{proof} The ``only if'' part is trivial. To prove the ``if'' part, assume that  the space
$X$ has countable tightness, countable ofan tightness and $X$ has a $\sigma$-locally finite Pytkeev$^*$ network $\mathcal N$. By Corollary~\ref{c:T1P*<=>wP}, the Pytkeev$^*$ network $\mathcal N$ in the countably tight space $X$ is an $\w$-Pytkeev network. Being $\sigma$-locally finite, the family $\mathcal N$ is point-countable. By Corollary~\ref{c:pcP<=>wP}, $\mathcal N$ is a point-countable Pytkeev network in $X$, which implies that $X$ has the strong Pytkeev property. By Theorem~\ref{t:MSak}, the space $X$ is first-countable and hence has countable fan tightness. Now  Theorem~4.5 of \cite{GK15b} implies that the space $X$ is metrizable.
\end{proof}

\section{Stability properties of the class of spaces with the strong Pytkeev$^*$ network.}

In this section we establish some stability properties of the class  of topological spaces with the strong Pytkeev$^*$ property.
It is easy to see that this class is local and closed under taking subspaces. A class $\C$ of topological spaces is {\em local} if a topological space $X$ belongs to the class $\C$ if and only if each point $x\in X$ has a neighborhood $O_x\subset X$ that belongs to the class $\C$.

Stability properties of the class of spaces with the strong Pytkeev property were studied in \cite{BL}. We shall prove that some results of \cite{BL} can be extended to spaces with the strong Pytkeev$^*$ property. In particular, this concerns Theorem 2.2 of \cite{BL} on the strong Pytkeev property in function spaces. To formulate this theorem and its modification, we need to recall some definitions from \cite{BL}.

Let $X$ be a topological space. A family $\I$ of compact subsets of $X$ is called {\em an ideal of compact sets} if $\bigcup\I=X$ and for any sets $A,B\in\I$ and any compact subset $K\subset X$ we get $A\cup B\in\I$ and $A\cap K\in\I$.

For an ideal $\I$ of compact subsets of a topological space $X$ and a topological space $Y$ by $C_\I(X,Y)$ we shall denote the space $C(X,Y)$ of all continuous functions from $X$ to $Y$, endowed with the {\em $\I$-open topology} generated by the subbase consisting of the sets
$$[K;U]=\{f\in C_\I(X,Y):f(K)\subset U\}$$where $K\in\I$ and $U$ is an open subset of $Y$.

If $\I$ is the ideal of compact (resp. finite) subsets of $X$, then the $\I$-open topology coincides with the compact-open topology (resp. the topology of pointwise convergence) on $C(X,Y)$. In this case the function space $C_\I(X,Y)$ will be denoted by $C_k(X,Y)$ (resp. $C_p(X,Y)$).

We shall be interested in detecting function spaces $C_\I(X,Y)$ possessing the strong Pytkeev property. For this we should impose some restrictions on the ideal $\I$. Following \cite{Ban} and \cite{BL} we define an ideal $\I$ of compact subsets of a topological space $X$ to be {\em discretely-complete} if for any compact subset $A,B\subset X$ such that $A\setminus B$ is a countable discrete subspace of $X$ the inclusion $B\in\I$ implies $A\in\I$.

It is clear that the ideal of all compact subsets of $X$ is discretely-complete. More generally, for any infinite cardinal $\kappa$ the ideal $\I$ of compact subsets of cardinality $\le\kappa$ in  $X$ is discretely-complete. On the other hand, the ideal of finite subsets of $X$ is discretely-complete if and only if $X$ contains no infinite compact subset with finite set of non-isolated points.

Let us recall \cite[\S11]{Gru} that a family $\mathcal N$ of subsets of a topological space is a {\em $k$-network} in $X$ if for any open set $U\subset X$ and compact subset $K\subset U$ there is a finite subfamily $\F\subset \mathcal N$ such that $K\subset\bigcup\F\subset U$.

Regular topological spaces with countable $k$-network are called {\em $\aleph_0$-spaces}. Such spaces were introduced by E.~Michael \cite{Mi} who proved that for any $\aleph_0$-spaces $X,Y$ the function space $C_k(X,Y)$ is an $\aleph_0$-space.  In \cite[2.2]{Ban} this result of Michael was extended to $\mathfrak P_0$-spaces: {\em for an $\aleph_0$-space $X$ and a $\mathfrak P_0$-space $Y$ the function space $C_k(X,Y)$ is a $\mathfrak P_0$-space}. By Theorem 2.2 of \cite{BL}, if a topological space $Y$ has a countable Pytkeev network at some point $y\in Y$, then for every $\aleph_0$-space $X$ the function space $C_k(X,Y)$ has a countable Pytkeev network at the constant function $\bar y:X\to\{y\}\subset Y$. A similar result holds also for countable Pytkeev$^*$ networks.

\begin{theorem}\label{t:function} Let $X$ be an $\aleph_0$-space (more generally, a Hausdorff space with countable $k$-network) and  $\I$ be a discretely-complete ideal of compact subsets of $X$. If a topological space $Y$ has the strong Pytkeev$^*$ property  at a point $y\in Y$, then the function space $C_\I(X,Y)$ has the strong Pytkeev$^*$ property at the constant function $\bar y:X\to\{y\}\subset Y$.
\end{theorem}

\begin{proof}
Let $\mathcal K$ be a countable $k$-network on the space $X$ and $\mathcal P$ be a countable Pytkeev$^*$ network at the point $y_0\in Y$ of the space $Y$. We lose no generality assuming that the networks $\mathcal K$ and $\mathcal P$ are closed under finite unions and finite intersections and each set $P\in\mathcal P$ contains the intersection $\ddot y$ of all neighborhoods of the point $y$ in $Y$.

For two subsets $K\subset X$ and $P\subset Y$ let
$$[K;P]=\{f\in C_\I(X,Y):f(K)\subset P\}\subset C_\I(X,Y).$$
We claim that the countable family
$$[\kern-1.5pt[\mathcal K;\mathcal P]\kern-1.5pt]=\{[K;P]:K\in\K,\;P\in\mathcal P\}$$
is an $\w$-Pytkeev network at the constant function $\bar y:X\to\{y\}\subset Y$ in the function space $C_\I(X,Y)$.

Given a countable subset $\A\subset C_\I(X,Y)$ accumulating at $\bar y$ and a neighborhood $O_{\bar y}\subset C_\I(X,Y)$ of $\bar y$ we need to find a set $\F\in[\kern-1.5pt[\K,\PP]\kern-1.5pt]$ such that $\F\subset O_{\bar y}$ and $\A\cap \F$ is infinite.

We lose no generality assuming that $\A\subset O_{\bar y}$ and the neighborhood $O_{\bar y}$ is of basic form $O_{\bar y}=[C;U]$ for some compact set $C\in\I$ and some open neighborhood $U\subset Y$ of the point $y$.

Since the space $X$ has countable network, the subspace $Z=\ddot y\cup \bigcup_{f\in\A}f(X)$ of $Y$ has a countable network and hence is countably tight. It follows that the subfamily  $\mathcal P'=\{P\cap Z:P\in\mathcal P,\;P\subset U\}$ is a countable Pytkeev$^*$ network at $y$ in the countably tight space $Z$. By Corollary~\ref{c:ct+P*=>P}, $\mathcal P'$ is a Pytkeev network at $y$ in  $Z$.

By the proof of Theorem 2.2 \cite{BL}, there exists a set $K\in\K$ with $C\subset K$ and a set $P'\in\mathcal P'$ such that the set $[K;P']\subset C_\I(X;Z)$ has infinite intersection with the set $\A$. Find a set $P\in\mathcal P$ such that $P'=P\cap Z$ and $P\subset U$. Then the set $[K;P]\supset [K;P']$ is contained in the neighborhood $[C;U]=O_{\bar y}$ of $\bar y$ and has infinite intersection with the set $\A$. This means that $\mathcal P$ is a countable $\w$-Pytkeev network at $\bar y$ in the functuon space $C_\I(X,Y)$. By Lemma~\ref{l:wP=>P*}, the space $C_\I(X,Y)$ has a countable Pytkeev$^*$ network at $\bar y$ and hence $C_\I(X,Y)$ has the strong Pytkeev$^*$ property at $\bar y$.
\end{proof}

Since the ideal of all compact subsets of a space $X$ is discretely-complete, Theorem~\ref{t:function} implies:

\begin{corollary}\label{c:main}  Let $X$ be a Hausdorff space with countable $k$-network. If a topological space $Y$ has the strong Pytkeev$^*$  property at a point $y\in Y$, then the function space $C_k(X,Y)$ has the strong Pytkeev$^*$ property at the constant function $\bar y:X\to\{y\}\subset Y$.
\end{corollary}

At presence of topological homogeneity of the function space $C_\I(X,Y)$, Theorem~\ref{t:function} allows to establish the Pytkeev$^*$ property at each point of $C_\I(X,Y)$ (not only at a constant function).
Let us recall that a topological space $X$ is {\em topologically homogeneous} if for any points $x,y\in X$ there is a homeomorphism $h:X\to X$ such that $h(x)=y$.

\begin{corollary}\label{c:homoPyt} Let $X$ be a Hausdorff space with a countable $k$-network, $\I$ be a discretely-complete ideal of compact subsets of $X$, and $Y$ be a topological space with the strong Pytkeev$^*$ property at some point $y\in Y$. If the function space $C_\I(X,Y)$ is topologically homogeneous, then $C_\I(X,Y)$ has the strong Pytkeev property.
\end{corollary}

A topological space $X$ is called {\em rectifiable} if for some point $e\in X$ there exists a homeomorphism $h:X\times X\to X\times X$ such that $h(x,e)=(x,x)$ and $h(\{x\}\times X)=\{x\}\times X$ for all $x\in X$. A typical example of a rectifiable space is any topological group $G$. In this case the homeomorphism $h:G\times G\to G\times G$, $h:(x,y)\mapsto (x,xy)$, witnesses that the space $G$ is rectifiable. It is known \cite{Gul} that a rectifiable space $X$ is metrizable if and only if $X$ is first countable and satisfies the separation axiom $T_0$. More information on rectifiable spaces can be found in \cite{Gul}, \cite{Ar}, \cite{Us}, \cite{BR}, \cite{Ban}, \cite{BL}.

\begin{corollary}\label{c:rec} Let $X$ be a Hausdorff space with a countable $k$-network and $\I$ be a discretely-complete ideal of compact subsets of $X$. For any rectifiable space $Y$ with the strong Pytkeev$^*$ property the function space $C_\I(X,Y)$ has the strong Pytkeev$^*$ property.
\end{corollary}

\begin{proof} By Corollary 4.4 \cite{BL}, the function space $C_\I(X,Y)$ is rectifiable and hence is topologically homogeneous. By Corollary~\ref{c:homoPyt}, the space $C_\I(X,Y)$ has the strong Pytkeev$^*$ property.
\end{proof}

Now we shall apply Theorem~\ref{t:function} to prove that the class of topological spaces with the strong Pytkeev$^*$ property is closed under countable Tychonoff products and countable small box-products of pointed spaces.

By a {\em pointed space} we understand a topological space $X$ with a distinguished point, which will be denoted by $*_X$.

By the {\em Tychonoff product} of pointed topological spaces $(X_\alpha,*_{X_\alpha})$ we understand the Tychonoff product $\prod_{\alpha\in A}X_\alpha$ with the distinguished point $(*_{X_\alpha})_{\alpha\in A}$. The {\em box-product} $\square_{\alpha\in A}X_\alpha$ of the spaces $X_\alpha$, $\alpha\in A$, is their Cartesian product $\prod_{\alpha\in A}X_\alpha$ endowed with the box-topology generated by the products $\prod_{\alpha\in A}U_\alpha$ of open sets $U_\alpha\subset X_\alpha$, $\alpha\in A$.

The subset $$\cbox_{\alpha\in A}X_\alpha=\big\{(x_\alpha)_{\alpha\in A}\in\square_{\alpha\in A}X_\alpha:\{\alpha\in A:x_\alpha\ne *_{X_\alpha}\}\mbox{ is finite}\big\}$$of the box-product $\square_{\alpha\in A}X_\alpha$ is called the {\em small box-product} of the pointed topological spaces $X_\alpha$, $\alpha\in A$. It is a pointed topological space with distinguished point $(*_{X_\alpha})_{\alpha\in A}$.

The subspace
$$\coprod_{\alpha\in A}X_\alpha =\big\{(x_\alpha)_{\alpha\in A}\in\prod_{\alpha\in A}X_\alpha:\big|\{\alpha\in A:x_\alpha\ne *_{X_\alpha}\}\big|\le1\big\}$$
of the Tychonoff product $\coprod_{\alpha\in A}X_n$ is called the {\em Tychonoff bouquet} of the pointed spaces $X_\alpha$, $\alpha\in A$. The same set $\coprod_{\alpha\in A}X_\alpha$ endowed with the box topology inherited from $\square_{\alpha\in A}X_\alpha$ will be called the {\em box-bouquet} of the pointed topological spaces $X_\alpha$, $\alpha\in A$, and will be denoted by $\bigvee_{\alpha\in A}X_\alpha$.

\begin{theorem} If $X_n$, $n\in\w$, are pointed topological spaces with the strong   Pytkeev$^*$ properties at their distinguished points, then
the Tychonoff bouquet $\coprod_{n\in\w}X_n$ and the Tychonoff product $\prod_{n\in\w}X_n$ both have the strong Pytkeev$^*$ property at their distinguished point $(*_{X_n})_{n\in\w}$.
\end{theorem}

\begin{proof} For every $n\in\w$, fix a countable Pytkeev$^*$ network $\mathcal N_n$ at the distinguished point $*_{X_n}$ of the space $X_n$. Identify each space $X_n$, $n\in\w$, with the subspace
 $\big\{(x_k)_{k\in\w}\in\coprod_{k\in\w}X_k:\forall k\ne n\;\; (x_k=*_{X_k})\big\}$ of the Tychonoff bouquet $X=\coprod_{k\in\w}X_k$  and observe that the family $$\mathcal N=\big(\bigcup_{n\in\w}\mathcal N_k\big)\cup \big\{\textstyle{\bigcup_{k\ge n}X_k:n\in\w}\big\}$$ is a countable Pytkeev$^*$ network at the distinguished point $*_X$ of $X$, which means that the space $X$ has the strong Pytkeev$^*$ property at $*_X$.

Observe that the Tychonoff product $\prod_{n\in\w}X_n$ can be identified with the (pointed) subspace $\{f\in C_k(\w,X):\forall n\in\w\;\;f(n)\in X_n\}$ of the function space $C_k(\w,X)$ whose distinguished point is the constant function $f_*:\w\to\{*_X\}\subset X$. By Corollary~\ref{c:main}, the space $C_k(\w,X)$ has the strong Pytkeev$^*$ property at its distinguished point $f_*$. Then the pointed subspace $\prod_{n\in\w}X_n$ of  $C_k(\w,X)$ has the strong Pytkeev$^*$ property at its distinguished point too.
\end{proof}

In the same way we can prove the preservation of the strong Pytkeev$^*$ property by small box-products.

\begin{proposition} If $X_n$, $n\in\w$, are pointed topological $T_1$-spaces with the strong Pytkeev$^*$ property at their distinguished points, then
the box-bouquet $\bigvee_{n\in\w}X_n$ and the small box-product $\cbox_{n\in\w}X_n$ both have the strong Pytkeev$^*$ property at their distinguished points.
\end{proposition}

\begin{proof} For every $n\in\w$ fix a countable Pytkeev$^*$ network $\mathcal N_n$ at the distinguished point $*_{X_n}$ of the space $X_n$. Identify each space $X_n$ with the subspace
 $\big\{(x_k)_{k\in\w}\in\textstyle{\bigvee_{k\in\w}}X_k:\forall k\ne n\;\; (x_k=*_{X_k})\big\}$ of the box-bouquet $X=\bigvee_{k\in\w}X_k$.
It is easy to check that the countable family $$\mathcal N=\bigcup_{n\in\w}\mathcal N_k$$ is a Pytkeev$^*$ network at the distinguished point $*_X$ of $X$, which means that $X$ has the strong Pytkeev$^*$ property at $*_X$.

Let $\w+1=\w\cup\{\infty\}$ be the one-point compactification of the (discrete) space $\w$ of non-negative integers. Observe that the small box-product $\cbox_{n\in\w}X_n$ is homeomorphic to the (pointed) subspace
$$\{f\in C_k(\w+1,X):f(\infty)=*_X\mbox{ and $\forall n\in\w$ $f(n)\in X_n$}\}$$of the function space $C_k(\w+1,X)$.
By Corollary~\ref{c:main}, the function space $C_k(\w+1,X)$ has the strong Pytkeev$^*$  property at the constant function $f_*:\w\to\{*_X\}$, which implies that the subspace $\cbox_{n\in\w}X_n$ of $C_k(\w+1,X)$ has the strong Pytkeev$^*$ property at its distinguished point $(*_{X_n})_{n\in\w}$.
\end{proof}

\end{document}